\documentclass[11pt]{amsart}

\usepackage{amsmath, amsfonts, amssymb, mathrsfs, euscript, amsthm}
\usepackage{algorithm}
\usepackage{algpseudocode}
\usepackage[usenames,dvipsnames,svgnames,table]{xcolor}

\usepackage{url}
\usepackage{hyperref}
\hypersetup{urlcolor=cyan,citecolor=blue,colorlinks=true}

\usepackage{setspace}  
\usepackage{breqn}
\usepackage[USenglish]{babel}
\usepackage{hyphenat}
\usepackage{multirow}

\newtheorem*{observation}{Observation}

\newtheorem{theorem}{Theorem}
\newtheorem{lemma}[theorem]{Lemma}
\newtheorem{corollary}[theorem]{Corollary}

\newcommand{\dlg}{\mathrm{dlg}}

\title[Semi-Primitive Roots and the Discrete Logarithm Modulo $2^{k}$]{Semi-Primitive Roots and the Discrete Logarithm Modulo $2^{k}$}

\author[]{Bianca Sosnovski}
\address{Department of Mathematics and Computer Science, Queensborough Community College, City University of New York, USA}
\email{bsosnovski@qcc.cuny.edu}

\begin{document}

\maketitle
 
 \begin{abstract}
We establish a connection between semi-primitive roots of the multiplicative group of integers modulo $2^{k}$ where $k\geq 3$, and the logarithmic base in the algorithm introduced by Fit-Florea and Matula (2004) for computing the discrete logarithm modulo $2^{k}$. Fit-Florea and Matula used properties of the semi-primitive root 3 modulo $2^{k}$ to obtain their results and provided a conversion formula for other possible bases. We show that their results can be extended to any semi-primitive root modulo $2^{k}$ and also present a generalized version of their algorithm to find the discrete logarithm modulo $2^{k}$. Various applications in cryptography, symbolic computation, and others can potentially benefit from higher precision hardware integer arithmetic. The algorithm is suitable for hardware support of applications where fast arithmetic computation is desirable.

\end{abstract}

\section{Introduction}
\label{sec:introduction}

Fit-Florea and Matula \cite{fit} presented a digit serial algorithm for computing the discrete logarithm of a residue modulo $2^k$ that uses 3 as the logarithmic base. The algorithm is suitable for hardware support of applications where fast arithmetic computation is desirable. Their interest was particularly in algorithms for hardware implementations where $k$ = 64, 128, 256, 512, or 1024.

Various applications in cryptography, symbolic computation, combinatorial problems in geometry, and others can potentially benefit from higher precision hardware integer arithmetic. So enhancing hardware capabilities for modular integer arithmetic operations at high-precision arguments is desirable. 

An approach to enhance modular integer arithmetic is through parallelization of operations. For example, the residue number system (RNS) provides the ability to perform fast and parallel computations. It can be used to support high-speed arithmetic by operating in parallel channels without the need for exchanging information among the channels. RNS is an integer representation that consists of a base of co-prime moduli $(m_1, m_2, \ldots, m_n)$ used to split an integer $X$ into smaller integers $(x_1, x_2, \ldots, x_n)$ where $x_i$ is the residue $X \equiv x_i \bmod m_i$.

In RNS, addition, subtraction, and multiplication operations are efficient and can be done without carry propagation between residue digits. Furthermore, the residue digits are usually much smaller than the binary representation of $X$. The RNS representation allows building arithmetic units for large numbers as a set of small and fast circuits. However, other arithmetic operations such as number comparison, division, and modular reduction are complex in RNS \cite{tomczak,ss16}.

RNS is widely used in signal processing applications. In future signal processing, many processing tasks may be performed in the optical domain. Encryption is an essential aspect of information security. Due to its inherent high speed and parallelism, all-optical encryption is a promising approach to improve network security. Bakhtiar and Hosseinzadeh \cite{bh16} demonstrated the concept of optical RNS and presented a scheme for all-optical encryption/decryption using optical RNS arithmetic operations. More RNS potential and applications can be found in \cite{mohan}.

Since for cryptosystems, computations with large integers (or $\mathbb{F}_p$ elements) are required, hardware support is needed, and RNS properties can be exploited to speed up cryptographic computations. For example, RNS has been used to speed up computations with large operands for RSA in \cite{bi04,nm01,ss14}, for elliptic curve cryptography in \cite{lp09,abs12,bde13}, and for lattice base- cryptography in \cite{bemp14,schi20}.

Computations over the moduli in RNS, or channels, are independent of each other. Thus, hardware support for computations in RNS where one of the moduli is $2^k$ can potentially be helpful. For example, Tomczak \cite{tomczak} proposed the Hierarchical RNS (HRNS) with two levels where the numbers are represented as a set of residues modulo factors of $2^k\pm 1$ and  $2^k$, and the converters between the HRNS and the binary representation use RNS with moduli $(2^k-1, 2^k, 2^k+1)$.

Let  $x=x_{n-1}x_{n-2}\ldots x_{2} x_{1}x_{0}$ be a $n$-bit integer. Following \cite{szabo}, the modular notation $\lvert x \rvert _{2^k} = x_{k-1}x_{k-2}\ldots x_{2} x_{1}x_{0}$ is used to
denote the congruence notation  $x_{n-1}x_{n-2}\ldots x_{2} x_{1}x_{0} \equiv x_{k-1}x_{k-2}\ldots x_{2} x_{1}x_{0} \mod 2^k$, that is, the value of the standard low-order $k$-bit string for all $1 \leq k \leq n$.  The discrete logarithm can be  used in the representation of $x$ by a triple $(s,p,e)$ such that $x= \lvert (-1)^{s}2^{p}h^{e}\rvert_{2^{k}}$ with $s=\{0,1\}$, $0\leq e \leq 2^{k-2}-1$, $0\leq p < 2^{k}$, and the minimal exponent $e$ is defined as the discrete logarithm of $x$ modulo $2^k$. We denote $ \dlg_{(h,k)}(x)$ the discrete logarithm $e$ with respect to base $h$.

This type of discrete logarithm factorization of integers requires converting integers to the exponential triple. Benschop in \cite{benschop} employed such alternative representation for integers to reduce modular multiplications to modular additions and modular exponentiations to modular multiplications. The discrete logarithm factorization can be uniquely determined by factoring out the largest even power $2^p$ and then applying the discrete logarithm algorithm to the odd part  $\lvert (-1)^{s}h^{e}\rvert_{2^{k}}$.  

Fit-Florea and Matula's algorithm computes the discrete logarithm $e$ of $x$ for the special case of modulo $2^k$. It employs $O(k)$ dependent multiplications modulo $2^{k}$ due to the serial nature of the algorithm. 

The proposed algorithm has some similarities to the index calculus method. Still, the main difference is that it is a fully deterministic algorithm, and storing values computed beforehand is not needed \cite{fit}. Other algorithms for discrete logarithm computation exist, but they are more complex and consider an arbitrary modulo. Algorithms such as Pollard's $\rho$, Pohlig-Hellman, the index calculus method, and Shanks baby-step giant-step are super polynomial algorithms, and only Shanks baby-step giant-step is deterministic \cite{odl2000}. In general, these algorithms are not efficient for actual applications where there is the assumption that the discrete logarithm is hard to compute. The generalized version of the Fit-Florea and Matula's algorithm, which finds the discrete logarithm for the particular case modulo $2^k$, can potentially be used in hardware support of higher precision modular arithmetic.

The rest of this paper is organized as follows: Section \ref{spr} presents the definition of semi-primitive roots and the foundation for why it is possible to generalize the bit-serial discrete logarithm for bases other than 3 as in the case of the original version of the algorithm. In \cite{fit}, the authors noted that other bases could be considered, and a change of base formula is provided. But the results presented in this paper go further by offering a general version of the algorithm that uses any base, which is a semi-primitive root modulo $2^k$. In Sections \ref{inheritance} and \ref{props}, the fundamental properties for the serial determination of the discrete logarithms modulo $2^k$ are presented.   In  Section \ref{algorithm}, we present the results used to distinguish positive powers from negative powers of semi-primitive root $h$  and the generalized Fit-Florea and Matula's algorithm.

\section{Semi-primitive roots modulo $2^{k}$}
\label{spr}

Let $\mathbb{Z}_{n}^{*}$ be the multiplicative group of integers modulo $n$. Lee, Kwon and Shin  \cite{lee1,lee2} define an integer $h$ as a \emph{semi-primitive root} modulo $n$ if the order of $h$ in  $\mathbb{Z}_{n}^{*}$ is equal to $\phi(n)/2$ \footnote{$\phi(n)$ is the Euler's totient function}.

\begin{theorem} [Lee, Kwon and Shin, 2011]
Suppose that  $\mathbb{Z}_{n}^{*} \cong C_{2} \times C_{\phi(n)/2}$. Then there exist a semi-primitive root $h\in \mathbb{Z}_{n}^{*}$ such that 
\[\mathbb{Z}_{n}^{*}=\left\{\pm h^{i} \bmod n\ : \ i=0,\ldots, \frac{\phi(n)}{2}-1\right\}.\]
\end{theorem}

Since $\mathbb{Z}_{2^{k}}^{*} \cong \mathbb{Z}_{2} \times \mathbb{Z}_{2^{k-2}}$ for all $k\geq 3$, we can represent $\mathbb{Z}_{2^{k}}^{*} $ in terms of its semi-primitive roots as $\mathbb{Z}_{2^{k}}^{*}=\langle -1 \rangle \times  \langle h \rangle $.

We have the direct implication below from the results proven in \cite{lee1}.

\begin{corollary} \label{c1} For $k\geq 3$ and any semi-primitive root $h$ in $\mathbb{Z}_{2^{k}}^{*}$,
\[\mathbb{Z}_{2^{k}}^{*}=\{\pm h^{i} \bmod 2^{k}\ : \ i=0,\ldots,2^{k-2}-1\}.\]
\end{corollary}

Nathanson in \cite[Section~3.2]{nathanson} showed that 5 is a semi-primitive root modulo $2^{k} $ for $k\geq 3$, that is, 
\[\mathbb{Z}_{2^{k}}^{*}=\{\pm 5^{i} \bmod 2^{k} \ : \ i=0,\ldots,2^{k-2}-1\}.\]

Fit-Florea and Matula \cite{fit} used the semi-primitive root 3 modulo $2^{k}$ for $k\geq3$ as the base for their discrete logarithm algorithm.  They also noted that other bases can be considered and that any residue of the form $\beta=\pm 3^{2i+1} \bmod 2^{k}, i \in\{0,1,2,\ldots,2^{k-2}-1\}$ will generate the same residues modulo $2^k$ as the residues generated by the corresponding powers of 3. The discrete logarithm for different bases can be obtained by the formula $ \dlg_{(\beta,k)}(x)=\displaystyle \frac{ \dlg_{(3,k)}(x)}{ \dlg_{(3,k)}(\beta)}$.

Because of the algebraic properties of semi-primitive roots modulo $2^{k}$, we can extend their results to find the discrete logarithm modulo $2^{k}$ using any semi-primitive root in $\mathbb{Z}_{2^{k}}^{*}$ as the logarithmic base without the need for a change of base.

Following the notation presented in \cite{szabo}, the modular function $\lvert m \rvert_{2^{k}}=j$ represents the congruence relation $m\equiv j \bmod 2^{k}$ where $k\geq 3$ and $0\leq j\leq 2^{k}-1$. The multiplicative inverse $\lvert m^{-1}\rvert_{2^{k}}$ exists for all odd $m$ with $0 < m \leq 2^{k}-1$.

Half of the odd integers modulo $2^{k}$ can be expressed as positive powers of $h$, while the other half can be expressed as negative powers of $h$.

Any $k$-bit integer $x=x_{k-1}x_{k-2}\ldots x_{1}x_{0}$  can be represented by a triple $(s,p,e)$ such that $x=\lvert(-1)^{s}2^{p}h^{e}\rvert_{2^{k}}$ with $s=\{0,1\}$, $0\leq e \leq 2^{k-2}-1$ and $0\leq p < 2^{k}$. 

Similarly to what is suggested in \cite{fit} and \cite{fit2}, the discrete logarithm factorization $x=\lvert(-1)^{s}2^{p}h^{e}\rvert_{2^{k}}$ is uniquely determined by first factoring out the largest power $2^{p}$ dividing $x$ as the even part factor and employing the discrete logarithm algorithm to the odd part factor $\lvert(-1)^{s}h^{e}\rvert_{2^{k}}$. 

Some of the following results are presented in \cite{fit} for the specific base $h=3$, and their generalizations are included in this paper.

%


\section{The Digit Inheritance Property}
\label{inheritance}
Given an integer with binary representation $x=x_{n-1}x_{n-2}\ldots x_{2}x_{1}x_{0}$ then for $1\leq k\leq n-1$,
\[\lvert x\rvert_{2^{k}}=\lvert x_{n-1}x_{n-2}\ldots x_{0}\rvert_{2^{k}}=x_{k-1}x_{k-2}\ldots x_{0},\]
that is reduction modulo $2^{k}$ is obtained by simply truncating the leading portion of the bit string.

An integer operation $z=x\otimes y$ has the \emph{Digit Inheritance Property} if for all nonnegative integers $x$ and $y$,
\[\lvert z \rvert_{2^{k}}=\mid \lvert x\rvert_{2^{k}} \otimes \lvert y\rvert_{2^{k}} \mid_{2^{k}} \mbox{ for all } k\geq 1.\]

An integer function $z=f(x)$ has the  \emph{Digit Inheritance Property} if for all nonnegative integers $x$,
\[\lvert z \rvert_{2^{k}}=\mid f(\lvert x \rvert_{2^{k}})  \mid_{2^{k}}  \mbox{ for all } k\geq 1.\]

The Digit Inheritance Property states that for operations and functions with this property, the low-order $k$ bits of the input arguments determine the low-order $k$ bits of the output for all $k\geq 1$.

Integer addition and multiplication operations and the exponentiation function satisfy the Digit Inheritance Property.

\section{Properties of the discrete logarithm modulo $2^{k}$}
 \label{props}
 
In this section, we present mathematical results that will be used to generalize the  Fit-Florea and Matula's discrete logarithm algorithm. We generalize the results presented in \cite{fit}  for any logarithmic base modulo $2^k$. This is possible because it depends on the multiplicative order of the base so that it can be adapted to any semi-primitive root.

\begin{lemma}
Let $h$ be a semi-primitive root modulo $2^{k}$, $k\geq3$. For any odd residue $A$, either $A$ or its additive inverse $-A$ is congruent to some power of $h$ modulo $2^{k}$.
 \end{lemma}
\begin{proof}
Apply Corollary \ref{c1}.
\end{proof}

Without loss of generalization, from now on, let $A$ be an odd residue modulo $2^{k}$ that can be expressed as a positive power of a semi-primitive root $h$. 

 \begin{lemma}\label{l2}
Let $B=\lvert A^{-1} \rvert_{2^{k}}$.   Then $\dlg_{(h,k)}(A)+\dlg_{(h,k)}(B)=2^{k-2}$ for $k\geq3$.
 \end{lemma}
\begin{proof}
If $a=\dlg_{(h,k)}(A)$ and $b=\dlg_{(h,k)}(B)$, then $\lvert A \rvert_{2^{k}}=\lvert h^{a}\rvert_{2^{k}}$ and $\lvert B\rvert _{2^{k}}=\lvert h^{b}\rvert_{2^{k}}$. Because $h$ is a semi-primitive root modulo $2^{k}$, we have that $\lvert h^{2^{k-2}} \rvert_{2^{k}}= 1$.

Since $\lvert AB \rvert_{2^{k}}= 1$, then $\lvert AB \rvert_{2^{k}}= \lvert h^{2^{k-2}}\rvert_{2^{k}} $. We also have that $\lvert h^{a}h^{b}\rvert_{2^{k}}=  \lvert h^{a+b} \rvert_{2^{k}}$. It follows that $a+b=2^{k-2}$.
\end{proof}

If the discrete logarithm mod $2^{k}$ is known for $B$, which is the multiplicative inverse of $A$, then we can compute 
$\dlg_{(h,k)}(A)=2^{k-2}-\dlg_{(h,k)}(B)$.

We can describe the multiplicative inverse of $\lvert(-1)^{s} h^{e}\rvert_{2^{k}}$  as $\lvert (-1)^{s} h^{2^{k-2}-e}\rvert_{2^{k}}$.

\begin{lemma} \label{l3}
For $k>3$ and $h$ any semi-primitive root mod $2^{k}$, we have that $\lvert h^{2^{k-3}}\rvert_{2^{k}}=\lvert 2^{k-1}+1\rvert_{2^{k}}$.
 \end{lemma}

\begin{proof}
Let $A= \lvert 2^{k-1}+1\rvert_{2^{k}}$. We have that $A=A^{-1}$ since $\lvert A^{2}\rvert_{2^{k}} \equiv \lvert (2^{k-1}+1)^{2}\rvert_{2^{k}}\equiv \lvert 2^{2(k-1)}+2\cdot 2^{k-1}+1 \rvert_{2^{k}}\equiv \lvert 1 \rvert_{2^{k}}$.

From Lemma \ref{l2}, $\dlg_{(h,k)}(A)+\dlg_{(h,k)}(A^{-1})=2^{k-2} \Longrightarrow 2\cdot \dlg_{(h,k)}(A)=2^{k-2}$. Therefore, $\dlg_{(h,k)}(A)=2^{k-3}$.
\end{proof}

\begin{corollary}
$\dlg_{(h,k)}(2^{k-1}+1)=2^{k-3}$.
\end{corollary}


\begin{corollary}[Digit Inheritance of the Discrete Logarithm]
 The low-order $(i-2)$ bits of the discrete logarithm function $\dlg_{(h,k)}(x)$ depends only on the low-order $i$ bits of the argument $x$ for $3\leq i \leq k$.
\end{corollary}

%

We can apply Lemma \ref{l3} to compute the discrete logarithm modulo $2^{i}$ of residues $(2^{i-1}+1)\bmod 2^{i}$ for any $i$. 


\begin{lemma}\label{l4} 
\emph{For any $k>3$,}
\[
\begin{array}{ccc}
\dlg_{(h,k)}(A)  & =  & \dlg_{(h,k-1)}(A)  \\
  & \mbox{or}  &   \\
\dlg_{(h,k)}(A)  & =  & \dlg_{(h,k-1)}(A) +2^{k-3}
\end{array}
\]

\end{lemma}

\begin{proof}
Let $a^{\prime}=\dlg_{(h,k-1)}(A) $ and $a=\dlg_{(h,k)}(A) $. So $ \lvert  h^{a^{\prime}}\rvert_{2^{k-1}} =  \lvert A \rvert_{2^{k-1}} $ and  $\lvert h^{a}\rvert_{2^{k}} = \lvert A\rvert_{2^{k}}  $

Because of the Digit Inheritance Property, $h^{a}$ and $h^{a^{\prime}}$ have the same digits whose binary weights are $2^{k-2}, \ldots, 2^{1}, 2^{0}$.

If their digits with weight $2^{k-1}$ are the same then $\lvert h^{a} \rvert_{2^{k}}=\lvert h^{a^{\prime}}\rvert_{2^{k}}$ . Therefore, $a=a^{\prime}$.

If not the same, we must have that $\lvert h^{a^{\prime}}+2^{k-1}\rvert_{2^{k}}=\lvert h^{a}\rvert_{2^{k}}$. 

Since $\lvert h^{a^{\prime}}\times 2^{k-1}\rvert_{2^{k}}=\lvert 2^{k-1}\rvert_{2^{k}}$, we have that

\begin{equation}\label{eq1}
\lvert h^{a^{\prime}}(2^{k-1}+1)\rvert_{2^{k}}=\lvert h^{a^{\prime}}2^{k-1}+h^{a^{\prime}}\rvert_{2^{k}}=\lvert 2^{k-1}+h^{a^{\prime}}\rvert_{2^{k}}=\lvert h^{a}\rvert_{2^{k}}.
\end{equation}

Applying Lemma \ref{l3},

\begin{equation} \label{eq2}
\lvert h^{a^{\prime}}(2^{k-1}+1)\rvert_{2^{k}}=\lvert h^{a^{\prime}}h^{2^{k-3}}\rvert_{2^{k}}=\lvert h^{a^{\prime}+2^{k-3}}\rvert_{2^{k}}
\end{equation}

Comparing (\ref{eq1}) and (\ref{eq2}),  $\lvert h^{a}\rvert_{2^{k}}=\lvert h^{a^{\prime}+2^{k-3}}\rvert_{2^{k}}$. Therefore, $a=a^{\prime}+2^{k-3}$.
\end{proof}

This result allows for the computation of $\dlg_{(h,k)}(A)$ one bit at a time.

\section{The Digit-Serial Discrete Logarithm Algorithm}
\label{algorithm}

The following results distinguish the positive powers of $h$ from the negative ones.

\begin{lemma}\label{l5} 
Let $A$ be an odd positive integer with $\lvert A \rvert_{2^{k}}=a_{k-1}a_{k-2}\ldots a_{2}a_{1}1$ then:
\begin{itemize}
\item[i)] \emph{If $A \equiv 1 \mod 4$ then $a_{1}=0$.}
\item[ii)] \emph{If $A \equiv 3 \mod 4$ then $a_{1}=1$.
}\end{itemize}
\end{lemma}

\begin{proof}
i) If $A \equiv 1 \mod 4$ then $A=1 + 4q$ for some integer $q\geq 0$. Because of the Digit Inheritance Properties of the addition and multiplication modulo $2^{k}$, 
\begin{align} 
\lvert A \rvert_{2^{k}} &=\lvert 1\rvert_{2^{k}} +\lvert \lvert 4\rvert_{2^{k}} \cdot \lvert q\rvert_{2^{k}}\rvert_{2^{k}}\nonumber \\
a_{k-1}a_{k-2}\ldots a_{2}a_{1}1 &=  0 \ldots 0001 +  \lvert 0 \ldots 0100 \times q_{k-1}q_{k-2} \ldots q_{1}q_{0}\rvert_{2^{k}}\nonumber \\
 &=  0 \ldots 0001+ q_{k-3}q_{k-4} \ldots q_{1}q_{0}00\nonumber \\
  &=  q_{k-3}q_{k-4} \ldots q_{1}q_{0}01
 \end{align} Thus, $a_{1}=0$.

ii) Similar can be shown for the case $A \equiv 3 \mod 4$ since $\lvert A\rvert_{2^{k}} = \lvert 3\rvert_{2^{k}} +\lvert \lvert 4\rvert_{2^{k}} \cdot \lvert q \rvert_{2^{k}}\rvert _{2^{k}}$. It follows from $\lvert 3\rvert_{2^{k}}=0 \ldots 0011$  that $\lvert A\rvert_{2^{k}}= a_{k-3}a_{k-4} \ldots a_{2}a_{1}a_{0}11$.
Therefore, $a_{1}=1$.
\end{proof}

%

\begin{theorem}\label{main theorem}
Let $h$ be a semi-primitive root mod $2^{k}$, $k\geq 3$. For all positive powers of $h$ with bit string $\lvert A\rvert_{2^{k}}=a_{k-1}a_{k-2}\ldots a_{2}a_{1}1$ we have that:
\begin{itemize}
\item[i)] If $h \equiv 1 \mod 4$ then $a_{1}=0$.
\item[ii)] If $h \equiv 3 \mod 4$ then $a_{2}=0$.
\end{itemize}
\end{theorem}

\begin{proof}
i) If  $h\equiv 1 \mod 4$ then $h^{i}\equiv 1 \mod 4$ for any $i\in \mathbb{N}$. By lemma (\ref{l5}), we must have $a_{1}=0$.

ii) If $h\equiv 3 \mod 4$ then $h^{2}\equiv 1 \mod 4$. So the powers of $h$ are of the form $h^{2 i}\equiv 1 \mod 4$ or $h^{2 i +1}\equiv 3 \mod 4$ for all $i\in \mathbb{N}$. Both 1 and 3 have binary representation with $a_{2}=0$.
\end{proof}


\begin{observation}
Let $\lvert h\rvert_{2^{k}}=h_{k-1}h_{k-2}\ldots h_{2}h_{1}1$ be a semi-primitive root modulo $2^{k}$. It follows from Lemma (\ref{l5}) and Theorem (\ref{main theorem}) that                 
for all positive powers of $h$ modulo $2^{k}$ whose binary representation is $\lvert A\rvert_{2^{k}}=a_{k-1}\ldots a_{2}a_{1}1$, 
$ \mbox{if } h_{1}+1= \begin{cases} 1 & \mbox{ then } a_{1}=0\\
                                       2 & \mbox{ then } a_{2}=0 \end{cases} $ .
                                      \end{observation}
 

 The bit strings representing the negative powers of $h$ are two's complements of the positive powers. In the case where $A$ is a negative power of $h$, if $h\equiv 1 \mod 4$  then $a_{1}=1$, and if $h\equiv 3 \mod 4$ then $a_{2}=1$.

This was noted in \cite{fit} for the semi-primitive root 3 where positive powers of $3 \bmod 2^{k}$ have binary digit $a_{2}=0$, while the negative powers of $3 \bmod 2^{k}$ have binary digit $a_{2}=1$.

In summary, one can identify if an integer $A$ modulo $2^{k}$ is a negative power of $h$ by checking if it satisfies one of the following conditions: 

i) $h_{1}=0$  and $a_{1}=1$

ii) $h_{1}=1$  and $a_{2}=1$

 \vspace{2mm}

Because of the Digit Inheritance Property for the residues modulo $2^{k}$, $\lvert h\rvert_{2^{k}}=h_{k-1}h_{k-2}\ldots h_{2}h_{1}1$ have the the same least significant digits as $\lvert h\rvert_{2^{3}}=h_{2}h_{1}1$. Since $\lvert h^{0}\rvert_{2^{3}}=001$ and $\lvert -h^{0}\rvert_{2^{3}}=111$ for any $h$,  we have that either $\lvert h\rvert_{2^{3}}=011$ or $\lvert h\rvert_{2^{3}}=101$.
We check the binary digits of $\lvert h\rvert_{2^{3}}$ to see which of the bits between  $h_{1}$ and $h_{2}$ is zero. From this, we can determine if an odd integer $A$  for which we want to find $\dlg_{(h,k)}(A)$ is a positive or negative power of $h \bmod 2^{k}$ for any $k\geq 3$. 

For the Algorithm \ref{dlg} below, $\lvert B\rvert_{2^{k}}=\lvert A^{-1}\rvert_{2^{k}}$ and $b=\dlg_{(h,k)}(B)$.

\begin{algorithm}[H]
	\caption{Generalized Fit-Florea and Matula DLG Digit-Serial Algorithm}\label{dlg}
\begin{algorithmic}[1]
\Require Odd integer $\lvert A\rvert_{2^{k}}=a_{k-1}a_{k-2}\ldots a_{2}a_{1}1$ \ \ \ \ \ \ \ \ \ \ \ \break Semi-primitive root $\lvert h\rvert_{2^{k}}=h_{k-1}h_{k-2}\ldots h_{2}h_{1}1$ 
\Ensure The factorization $(s,e)$ of $A$ as $A= \lvert (-1)^{s}h^{e}\rvert_{2^k}$
\Statex
 	\State $B:=1$;  \label{comment1}
	\State $b:=0$;   \label{comment2} 
	\If{($h_{1}=0$ and $a_{1}=1$) or ($h_{1}=1$ and $a_{2}=1$)}   \label{comment3}
		\State $P:=\lvert -A\rvert_{2^{k}}$;  \label{comment4}
		\State $s:=1$;
	\Else
		\State  $P:=\lvert A\rvert_{2^{k}}$;
		\State $s:=0$; 
	\EndIf\label{ifloop1}
	\If{$\lvert P\rvert_{2^{3}}=\lvert h\rvert_{2^{3}}$} \label{comment5}
		\State $B:=\lvert h\rvert_{2^{k}}$;
		\State $b:=1$;
	\EndIf\label{ifloop2}
	\State $P:=\lvert P\times B\rvert_{2^{k}}$: \label{comment6}
	\For{$i$ from 3 to $k-1$}\label{comment7}
		\If{$p_{i}=1$}
			\State $b:=b+2^{i-2}$;  \label{comment8}
			\State $B:=\lvert B\times h^{2^{i-2}}\rvert_{2^{k}}$;
			\State $P:=\lvert P\times h^{2^{i-2}}\rvert_{2^{k}}$;  \label{comment9}
		\EndIf\label{ifloop3}
	\EndFor\label{forloop1}
\State $e=\lvert 2^{k-2}-b\rvert_{2^{k-2}}$
\State \textbf{return} $(s,e)$ 

\end{algorithmic}
\end{algorithm}

We note the following about algorithm \ref{dlg}:

\begin{itemize}
\item Lines \ref{comment1} and \ref{comment2}  correspond to the binary representation of 1 with $k$ bits and the binary representation of 0 with $k-2$ bits, respectively.
\item The condition in line \ref{comment3} identifies the case where $A$ is a negative power of $h$.
\item Line \ref{comment4} represents the two's complement of $A$.
\item There are only two possibilities for $P$, either $\lvert P\rvert_{2^{3}}=001$ or $\lvert P\rvert_{2^{3}}=\lvert h\rvert_{2^{3}}$, which is what being tested in line \ref{comment5}.
\item Line \ref{comment6} performs the binary multiplication of $k$-digit numbers.
\item Line \ref{comment8}  flips the $(i-1)$-th bit of $b$.
\item The loop from line \ref{comment7} to line \ref{comment9} employs $O(k)$ dependent multiplications modulo $2^{k}$.
\end{itemize}

\section{Conclusion}
We presented a generalized version of Fit-Florea and Matula's algorithm. The algorithm finds the discrete logarithm modulo $2^{k}$ whose base is a semi-primitive root. With the present results, the discrete logarithm can be computed directly for bases other than 3, and the algorithm remains efficient, requiring $O(k)$ multiplications modulo $2^{k}$.

%
%

%
%
%
%
%


\end{document}